\documentclass{amsart}

\usepackage{graphicx}

\newtheorem{theorem}{Theorem}[section]

\newtheorem{proposition}[theorem]{Proposition}

\theoremstyle{definition}
\newtheorem{definition}[theorem]{Definition}

\newtheorem{notation}{Notation}
\theoremstyle{remark}
\newtheorem{remark}[theorem]{Remark}

\numberwithin{equation}{section}

\begin{document}

\title[Min-max hypersurfaces]{
Orientability of min-max hypersurfaces in manifolds of positive Ricci curvature}
\author{Alejandra Ram\'irez-Luna}
\address{Departmento de Matem\'aticas, Universidad del Valle,‎ Calle 13 No. 100-00, Cali, Colombia.}
\email{maria.ramirez.luna@correounivalle.edu.co}
\thanks{The first author was supported by IMU and TWAS}

\date{July  29, 2019}

\dedicatory{}


\begin{abstract}
Let $M^{n+1}$ be an orientable compact Riemannian manifold with positive Ricci curvature. We prove that the Almgren-Pitts width of $M^{n+1}$ is achieved by an orientable index $1$ minimal hypersurface with multiplicity $1$ and optimal regularity. This extends to dimensions $n+1\geq 8$ the results of Ketover-Marques-Neves \cite{ketover2016catenoid} and X. Zhou \cite{zhou2017min}.

\end{abstract}

\maketitle

\section{Introduction}
Let $M^{n+1}$ be an orientable compact Riemannian manifold. A problem of great interest is the search of minimal hypersurfaces in $ M^{n+1} $. 
The min-max theory is a method established by F. Almgren and J. Pitts to produce minimal hypersurfaces.  They showed the existence of a smooth, closed, embedded minimal hypersurface in a Riemannian manifold when $2\leq n \leq 5$ (\cite{almgren1965theory},  \cite{pitts2014existence}) . Schoen and Simon \cite{schoen1981regularity} generalized it to the case of dimension $n\geq 6$ (when $n\geq 7$ this hypersurface might have a singular set of Hausdorff codimension greater than or equal to $7$). 
Let $Z_{n}(M^{n+1})$ be the space of integral cycles and $\Phi:[0,1]\rightarrow Z_{n}(M^{n+1})$ (which will be denoted by $\{\Phi_{s}\}_{s=0}^{1}$) be a sweepout. Denote by $\Pi$ all the mappings $\Lambda$ which are homotopic to $\Phi$ in $Z_{n}(M^{n+1})$. We define the width of $M^{n+1}$ as\\

\begin{center}
$\displaystyle L(\Pi):=\inf_{\Lambda \in \Pi}\{\max_{x\in[0,1]}\mathcal{H}^{n}(\Lambda(x))\}$.
\end{center}

\begin{theorem}\cite{pitts2014existence}\cite{schoen1981regularity}
Under the above conditions $L(\Pi)>0$, and there exists a stationary integral varifold $V$, whose support is a disjoint collection of connected, closed, singular, minimal hypersurfaces $\{\Sigma_{i}\}_{i=1}^{l}$, with singular sets of Hausdorff dimension no larger than $n-7$ (which may have multiplicity $m_{i}\in \mathbb{N}$), such that $V=\sum_{i=1}^{l}m_{i}[\Sigma_{i}]$, and 
\begin{center}
$\displaystyle \sum_{i=1}^{l}m_{i}\mathcal{H}^{n}(\Sigma_{i})=L(\Pi).$
\end{center}
\end{theorem}

The characterization of these hypersurfaces has been studied over the last years. Xin Zhou characterized the min-max hypersurface when $Ric_{g}>0$, first for  $2\leq n \leq 6$ and then for $n\geq 7$, after work of Marques-Neves \cite{marques2012rigidity} when $n=2$.

\begin{theorem}\cite{zhou2017min}
\label{xin2018}
Let $(M^{n+1},g)$ be any $(n+1)$ dimensional connected, closed, orientable Riemannian manifold with positive Ricci curvature. Then the min-max hypersurface $\Sigma$ is either:

\begin{itemize}
\item orientable of multiplicity one, with Morse index one and $\mathcal{H}^{n}(\Sigma)=A_{M}$;
\item or non-orientable of multiplicity two with $2\mathcal{H}^{n}(\Sigma)=A_{M}$
\end{itemize}
where

\begin{center}
$\displaystyle A_{M} =\inf_{\sigma \in S}\left\{ \begin{array}{lcc}
              \mathcal{H}^{n}(\sigma) &  {\rm if} \hspace{0.1cm} \Sigma \hspace{0.1cm} {\rm is} \hspace{0.1cm} {\rm orientable}\\
             2\mathcal{H}^{n}(\sigma) &  {\rm if} \hspace{0.1cm} \Sigma \hspace{0.1cm} {\rm is}\hspace{0.1cm}  {\rm non} \hspace{0.1cm} {\rm orientable}
           \end{array} 
   \right.$
\end{center} 
and $S$ is the set of singular hypersurfaces $\sigma$ with a singular set of Hausdorff co-dimension no less than $7$ (when $2\leq n \leq 6$, we can take $\sigma$ to be smooth) such that $\bar{\sigma}$ is connected, closed and minimal. 
\end{theorem}

\begin{remark}\label{J} In the proof of the last theorem we can see that $L(\Pi)=\mathcal{H}^{n}(\Sigma)$, if $\Sigma$ is orientable or $L(\Pi)=2\mathcal{H}^{n}(\Sigma)$, if $\Sigma$ is non-orientable. 
\end{remark}

Later on,  Ketover, Marques and Neves \cite{ketover2016catenoid} ruled out the second possibility in the last theorem when $3\leq n+1 \leq 7$ 

\begin{theorem}\cite{ketover2016catenoid}
\label{coda2016}
For $3\leq n+1 \leq 7$, the Almgren-Pitts width of an orientable Riemannian manifold $M^{n+1}$ with positive Ricci curvature is achieved by an orientable index $1$ minimal hypersurface with multiplicity $1$. 
\end{theorem}

In this paper we complete the characterization by showing that orientability also holds in high dimensions. Our main result is:

\begin{theorem}
\label{ours}
The Almgren-Pitts width of an orientable Riemannian manifold $M^{n+1}$ with positive Ricci curvature is achieved by an orientable index $1$ minimal hypersurface with multiplicity $1$ and optimal regularity.
\end{theorem}

\begin{remark}
If $M$ is non-orientable the Theorem \ref{ours} is true with orientable index 1 changed by $2-$sided index 1 and considering the Almgren-Pitts width with coefficients in $\mathbb{Z}_{2}$.
\end{remark}
The proof of Theorem \ref{ours} will be carried out by contradiction. By Theorem \ref{xin2018} and  Remark \ref{J} suppose the min-max hypersurface $\Sigma$ is non-orientable and $L(\Pi)=2\mathcal{H}^{n}(\Sigma)$. It will suffice to show the existence of a sweepout $\{\Lambda_{s}\}_{s=0}^{1}$ in $M^{n+1}$ such that

\begin{center}
 $\displaystyle\sup_{s\in[0,1]}\mathcal{H}^{n}(\Lambda_{s})< 2\mathcal{H}^{n}(\Sigma)$.
\end{center}
 We will start with the family given by Xin Zhou (see Theorem \ref{prop3.4xin2018}). We will modify it as in \cite{ketover2016catenoid} to obtain the desired sweepout.

In fact, the family will begin in the same way that Xin Zhou began the distance family $\{\Sigma_{t}\}$  until an appropriate time $\delta_{0}$.
At the fixed moment $\delta_{0}$, a cylinder of constant height $\delta_{0}$ will be opened up to an appropriate radius $ R $. Then,  with constant radius $ R $ the family given by Zhou will be continued but at every moment $ h $ the cylinder of radius $R$ and height $ h $ will appear and the interior of the intersection between the cylinder mentioned above and $\Sigma_{h}$ will be removed.

\textbf{Acknowledgments}
This paper was made possible thanks to a PhD scholarship (IMU Breakout Graduate Fellowship) from IMU and TWAS to the author. I  am very grateful for the excellent direction of Professor Fernando Cod\'a Marques.  This work was done while the author was visiting Princeton University as a VSRC. I am grateful to Princeton University for the hospitality. I am also thankful to the Department of Mathematics at Universidad del Valle for partial support my visit to Princeton.

\section{Preliminary concepts and theorems}
In this section we give some concepts and theorems needed along the paper. Let $(M^{n+1},g)$ be a connected, closed, orientable Riemannian manifold and $\Sigma$ a closed, embedded, orientable hypersurface in $M^{n+1}$ with singular set of Hausdorff co-dimension no less than $7$. We consider $M^{n+1}$  embedded in $\mathbb{R}^{N}$ for some $N$.

\begin{notation}
The following notations and definitions are given by Xin Zhou in section 2 in \cite{zhou2017min}.

\begin{itemize}
\item $C(M)$ is the space of sets $\Omega \subset M$ with finite perimeter.
\item $[[\Omega ]]$ is the corresponding integral current with the natural orientation when $\Omega \in C(M)$ or 
$\Omega$ is a closed, orientable hypersurface in $M$ with a singular set of Hausdorff dimension no larger that $n-7$. 
\item $\mathcal{M}$ and $\mathcal{F}$ are the mass norm and flat norm of the space of $k-$dimensional integral currents in $\mathbb{R}^{N}$ with support in $M^{n+1}$, respectively.

\end{itemize}
\end{notation}

\begin{notation}Let us denote,
\begin{itemize}
\item
$B_{t}(p):=\{x\in M; d(p,x)<t\}$.

\item
$\displaystyle C_{t,h}(p):=\{exp_{q}(s\eta(q)); q\in  \partial B_{t}(p)\cap\Sigma ; s \in [-h,h] \} $

\item $\displaystyle B_{t,h}(p):= \{exp_{q}(h\eta(q)):q\in B_{t}(p)\cap \Sigma \}\cup \{exp_{q}(-h\eta(q)):q\in B_{t}(p)\cap \Sigma \} $

\end{itemize}
where $\eta$ is a unit normal vector field defined in a neighborhood of $p$ in the regular part of $\Sigma$.

\end{notation}

\begin{definition}
If $\Sigma$ has singularities let us denote by $Sing(\Sigma)$ the set of singularities of $\Sigma$ and $S$ the set of connected, closed, minimal, hypersurfaces $\Sigma$ of $M^{n+1}$ with $Sing(\Sigma)$ of codimension no less than $7$. We say $\Sigma \in S$ has optimal regularity.  

\end{definition}
The following result was proven in \cite{zhou2017min} (see Proposition 3.6 of \cite{zhou2017min}). Let $\Sigma \in S$ be a non-orientable hypersurface in $M$. Consider the distance family  $\{\Sigma_{t}\}$ given by

\begin{center}
$\Sigma_{t}:=\{x\in M^{n+1}: dist(x,\Sigma)=t\}$.
\end{center}
where $dist(x,\Sigma)$ is the non-signed distance function.

\begin{proposition}\cite{zhou2017min}
\label{prop3.4xin2018}
Assume that $Ric_{g}>0$. For any $\Sigma \in S$ non-orientable, the distance family $\{\Sigma_{t}\}_{t\in [0,d(M)]}$ satisfies:

\begin{itemize}
\item $\Sigma_{0}=\Sigma$
\item $\mathcal{H}^{n}(\Sigma_{t})< 2\mathcal{H}^{n}(\Sigma)$, for all $0< t \leq d(M)$
\item When $t\rightarrow 0$, $\mathcal{H}^{n}(\Sigma_{t})\rightarrow 2\mathcal{H}^{n}(\Sigma)$, and $\Sigma_{t}$ converge smoothly to a double cover of $\Sigma$ in any open set $U\subset M\setminus Sing(\Sigma)$ with compact closure $\bar{U}$.
\end{itemize}
where $d(M)$ is the diameter of $M^{n+1}$.
\end{proposition}

\begin{remark} Notice that $dist(x,\bar{\Sigma})$ is achieved by a regular point of $\Sigma$. Also, in the proof of the last proposition we have the estimate
\label{cota}
\begin{equation}
\mathcal{H}^{n}(\Sigma_{t})\leq 2\cos^{n}(\sqrt{\kappa}t)\mathcal{H}^{n}(\Sigma),
\end{equation}
where $\kappa$ is a constant such that $Ric_{g}\geq n \kappa$.
\end{remark}

\begin{remark}
\label{scale}
Using a change of variables, we can consider $\{\Sigma_{t}\}_{t=0}^{d(M)}$ in the interval $[0,1]$ ($\{\Sigma_{s}\}_{s=0}^{1}$).
\end{remark}

\section{MAIN RESULT}
The reader can find in Section 4 of \cite{zhou2017min} the precise definitions of width and sweepouts in Almgren-Pitts min-max theory. 

\begin{theorem}
\label{maintheorem}
The Almgren-Pitts width of an orientable closed Riemannian manifold $M^{n+1}$ with positive Ricci curvature is achieved by an orientable index $1$ minimal hypersurface with multiplicity $1$ and optimal regularity.
\end{theorem}

\begin{proof} Because  of  Theorem \ref{coda2016} we can assume $n\geq 7$. We are going to proceed by contradiction. Let $\Sigma^{n}$ be the min-max hypersurface. By Theorem \ref{xin2018} and Remark \ref{J} we can suppose $\Sigma$ is a  closed embedded non-orientable minimal hypersurface such that $L(\Pi)=2\mathcal{H}^{n}(\Sigma)$. (Notice that $\Sigma$ might have a singular set of codimension at least $7$). The idea is to create a sweepout $\{ \Lambda_{s} \}_{s=0}^{1}$ in $M^{n+1}$ such that

\begin{equation}
\sup_{s\in[0,1]}\mathcal{H}^{n}(\Lambda_{s})< 2\mathcal{H}^{n}(\Sigma),
\end{equation}
which by definition of $L(\Pi)$ would be a contradiction. 

By Proposition \ref{prop3.4xin2018} and Remarks \ref{cota} and \ref{scale}, we have a sweepout $\{\Sigma_{r}\}_{r=0}^{1}$ of $M^{n+1}$ and $\delta>0$ such that $\Sigma_{0}=\Sigma$,

\begin{equation}
\label{control0}
\sup_{r\in[\delta,1]}\mathcal{H}^{n}(\Sigma_{r})< 2\mathcal{H}^{n}(\Sigma)  \hspace{0.3cm}\text{and}\hspace{0.3cm} \mathcal{H}^{n}(\Sigma_{h})\leq 2\mathcal{H}^{n}(\Sigma)-Ah^{2}  \hspace{0.3cm};\forall h\in [0,\delta],
\end{equation}
for some constant $A>0$. We are going to modify $\{\Sigma_{r}\}_{r=0}^{1}$ 
from $r=\delta_{0}$ to $r=0$ for some $\delta_{0}\leq \delta$ that will be specified later. At the instant $r=\delta_{0}$ we will start to open up a cylinder of height $\delta_{0}$ in a point $p\in \Sigma$ (away from the singularities of $\Sigma$) of radius $t=0$ to $t=R$, for an $R$ that will be given below.
We have the following Euclidean comparisons. There exists an $R>0$ so that for any $p\in \Sigma$ and $t\leq R$ there holds

\begin{equation}
ct^{n}\leq \mathcal{H}^{n}(\Sigma\cap B_{t}(p)) \leq Ct^{n}
\end{equation}
\begin{equation}
ct^{n-1}\leq \mathcal{H}^{n-1}(\Sigma \cap \partial B_{t}(p))\leq Ct^{n-1}.
\end{equation}
Also there exists $h_{0}>0$ so that whenever $h\leq h_{0}$ one has the next area bounds

\begin{equation}
\label{areas1}
cht^{n-1}\leq \mathcal{H}^{n}(C_{t,h}(p)) \leq Cht^{n-1}
\end{equation}
and

\begin{equation}
\label{areas2}
ct^{n}\leq \mathcal{H}^{n}(B_{t,h}(p))\leq Ct^{n}.
\end{equation}

Let $p\in\Sigma$ such that $B_{R}(p)\cap \Sigma$ does not contain singular points. Let us define

\begin{equation}
\label{defnewsurface}
\Lambda_{h,t}:= \{\Sigma_{h} \cup C_{t,h}(p)\}\setminus B_{t,h}(p).
\end{equation}
If follows from equations (\ref{control0}), (\ref{areas1}) and (\ref{areas2})   that

\begin{equation}
\label{4.8}
\mathcal{H}^{n}(\Lambda_{h,t})\leq 2\mathcal{H}^{n}(\Sigma)-Ah^{2}+Cht^{n-1}-ct^{n}.
\end{equation} 
The maximum value of the function $f(t):=Cht^{n-1}-ct^{n}$ is at the point

\begin{equation}
\label{4.9}
t=\frac{Ch(n-1)}{nc}.
\end{equation}
Therefore from inequality (\ref{4.8}) there is a $B>0$ (independent of $t$) such that  for all $t\leq R$

\begin{equation}
\label{4.10}
\mathcal{H}^{n}(\Lambda_{h,t})\leq 2\mathcal{H}^{n}(\Sigma)-Ah^{2}+Bh^{n}.
\end{equation}
Making $h_{0}$ smaller we have that for all $h\leq h_{0}$ and $t\leq R$

\begin{equation}
\label{4.11}
\mathcal{H}^{n}(\Lambda_{h,t})\leq 2 \mathcal{H}^{n}(\Sigma)-\frac{A}{2}h^{2}.
\end{equation}
Notice that in the previous step it was important to have $n>2$. When $n=2$ it is necessary to consider the catenoid estimate \cite{ketover2016catenoid}. 
Moreover, when $t=R$ we obtain from inequality (\ref{4.8})
\begin{equation}
\label{4.12}
\mathcal{H}^{n}(\Lambda_{h,R})\leq 2\mathcal{H}^{n}(\Sigma)-Ah^{2}+ChR^{n-1}-cR^{n}.
\end{equation}
Making $h_{0}$ smaller again so that $h_{0}\leq \frac{cR}{2C}$ we obtain for $h\leq h_{0}$

\begin{equation}
\label{4.13}
\mathcal{H}^{n}(\Lambda_{h,R})\leq 2\mathcal{H}^{n}(\Sigma)-Ah^{2}-\frac{c}{2}R^{n}.
\end{equation}
The last inequality tell us that opening the hole up to time $t=R$ we decrease area by a definite amount depending on $R$ and not on $h$.

Let $\delta_{0}=\min\{\delta, h_{0}\}$. Since $\Lambda_{\delta_{0},0}=\Sigma_{\delta_{0}}$, we can continue our sweepout by concatenating $\{ \Lambda_{\delta_{0},t}\}_{t=0}^{R}$ with  $\{ \Sigma_{r} \}_{r=\delta_{0}}^{1}$. From inequality (\ref{4.11}) we have control of $\mathcal{H}^{n}(\Lambda_{\delta_{0},t})$. In fact, for all $t \leq R$

\begin{equation}
\label{control1}
\mathcal{H}^{n}(\Lambda_{\delta_{0},t})\leq 2\mathcal{H}^{n}(\Sigma)-\frac{A}{2}\delta_{0}^{2}.
\end{equation} 
We continue our sweepout by concatenating $\{\Lambda_{h,R}\}_{h=\delta_{0}}^{0}$ to $\{ \Lambda_{\delta_{0},t}\}_{t=0}^{R}$. From inequality (\ref{4.13}) we also have control of $\mathcal{H}^{n}(\Lambda_{h,R})$. In fact, for all $h\leq \delta_{0}$ 

\begin{equation}
\label{control2}
\mathcal{H}^{n}(\Lambda_{h,R})\leq 2\mathcal{H}^{n}(\Sigma)-\frac{c}{2}R^{n}.
\end{equation}
Let us define the sweepout $\{\Lambda_{s}\}_{s=0}^{1}$ described above:

\begin{center}
$\displaystyle \Lambda_{s}=\left\{ \begin{array}{lcc}
              \Lambda_{3\delta_{0}s,R} &  {\rm if}  & 0\leq s \leq 1/3 \\
              \Lambda_{\delta_{0},-3Rs+2R} &  {\rm if} & 1/3 \leq s \leq 2/3\\
              \Sigma_{3(1-\delta_{0})s+3\delta_{0}-2} &  {\rm if} & 2/3\leq s \leq 1
             \end{array}
   \right.$
\end{center} 
By inequalities (\ref{control0}), (\ref{control1}) and (\ref{control2}), the sweepout $\{\Lambda_{s}\}_{s=0}^{1}$ is such that

\begin{equation}
\sup_{s\in [0,1]}\mathcal{H}^{n}(\Lambda_{s})< 2\mathcal{H}^{n}(\Sigma).
\end{equation}
Contradiction, which finishes the proof of Theorem \ref{maintheorem}.

\end{proof}

\bibliographystyle{plain}
\bibliography{mybib}

\end{document}